\documentclass[a4paper, 10pt]{article}

\usepackage{graphics,graphicx}
\usepackage{amssymb,amsmath}
\usepackage{enumitem}
\usepackage{color}
\usepackage{amsmath}
\usepackage{amsfonts}
\usepackage{amssymb}
\usepackage{fullpage}
\usepackage[capitalize]{cleveref}
\usepackage{color}

\newtheorem{theorem}{Theorem}
\newtheorem{rlt}{Theorem}

\newtheorem{case}{Case}
\usepackage{thmtools}
\newlist{lemmalist}{enumerate}{1}
\setlist[lemmalist]{label=(\roman{lemmalisti}),
	ref=\thelemma:$(\roman{lemmalisti})$,
	noitemsep}
\declaretheorem[
name=Lemma]{lemma}

\Crefname{lemmalist}{Lemma}{Lemmas}
\addtotheorempostheadhook[lemma]{\crefalias{lemmalisti}{lemmalist}}

\newtheorem{claim}{Claim}
\newtheorem{problem}{Problem}
\newtheorem{conjecture}{Conjecture}

\newtheorem{corolarry}{Corollary}

\def \no {\noindent}
 
 \def \sm {\setminus}
 \def \es {\emptyset}

\newenvironment{proof}[1][]%
{\noindent {\setcounter{equation}{0}\it Proof.
}{#1}{}}{\hfill$\Box$\vspace{0ex}}

\newenvironment{proof2}[1][]%
{\noindent{\setcounter{equation}{0}\it Proof.
}{#1}{}}{$\diamond$\vspace{0ex}}

\begin{document}

\title{Coloring ($P_5$, kite)-free graphs}
\author{
Shenwei Huang\thanks{College of Computer Science, Nankai University, Tianjin 300350, China.  E-mail:~shenweihuang@nankai.edu.cn. Supported by the National Natural Science Foundation of China (12171256) and Natural Science Foundation of Tianjin (20JCYBJC01190).}
\and
Yiao Ju\thanks{College of Computer Science, Nankai University, Tianjin 300350, China. E-mail:~2120190414@mail.nankai.edu.cn.}
\and
T. Karthick\thanks{Computer Science Unit, Indian Statistical Institute, Chennai Centre, Chennai 600029, India.
}
}
\date{\today}
\maketitle

\begin{abstract}
Let $P_n$  and $K_n$ denote the induced path and complete graph on $n$ vertices, respectively.
The {\em kite} is the graph obtained from a  $P_4$ by adding a vertex and making it adjacent to all vertices in the $P_4$ except one vertex with degree 1. A graph is ($P_5$, kite)-free if it has no induced subgraph isomorphic to a $P_5$ or a kite. For a graph $G$, the chromatic number of $G$  (denoted by $\chi(G)$) is the minimum number of colors needed to color the vertices of $G$ such that no two adjacent vertices receive the same color,   and the clique number of $G$ is the size of a largest clique in $G$.  Here, we are interested in the class of ($P_5$, kite)-free graphs with small clique number. It is known that  every ($P_5$,~kite, $K_3$)-free graph $G$ satisfies $\chi(G)\leq 3$, every ($P_5$,~kite, $K_4$)-free graph $G$ satisfies $\chi(G)\leq 4$, and that every ($P_5$,~kite, $K_5$)-free graph $G$ satisfies $\chi(G)\leq 6$. In this paper, we showed  the following:
\begin{itemize}\itemsep=0pt
\item Every ($P_5$,~kite, $K_6$)-free graph $G$  satisfies $\chi(G)\leq 7$.
  \item Every ($P_5$,~kite, $K_7$)-free graph $G$ satisfies $\chi(G)\leq 9$.
  \end{itemize}
We also give examples to show that the above bounds are tight.
\end{abstract}

\no{\bf Keywords}: Graph classes; $P_5$-free graphs; Chromatic number; Clique number.

\section{Introduction}
All graphs in this paper are finite and simple. For general graph theory notation which are not defined here we follow~\cite{West}.
Let $P_n$, $C_n$ and $K_n$ denote the path, cycle and complete graph on $n$ vertices, respectively.
We denote the complement graph of $G$ by $\overline{G}$. For two vertex disjoint  graphs $G$ and $H$, we write $G+H$ to denote the \emph{disjoint union} of $G$ and $H$, and $G\vee H$
to denote the graph obtained from $G+H$ by adding an edge between every vertex in $G$ and every vertex in $H$. For a positive integer $r$, we use $rG$ to denote the disjoint union of $r$ copies of $G$. A \emph{hole} is an induced cycle on four or more vertices. An \emph{antihole} is the complement graph of a hole. A hole or antihole is \emph{odd} (\emph{even}) if it has an odd  (even) number of vertices. A \emph{clique} (\emph{stable set}) in a graph $G$ is a set of pairwise adjacent (nonadjacent) vertices in $G$.
We say that a graph $G$ \emph{contains} a graph $H$ if $G$ has an induced subgraph that  is isomorphic to $H$. A graph $G$ is \emph{$H$-free} if it does not contain $H$. For a family $\mathcal{H}$ of graphs, $G$ is \emph{$\mathcal{H}$-free} if $G$ is $H$-free for every $H\in\mathcal{H}$. We write $(H_1,\ldots,H_n)$-free instead of $\{H_1,\ldots,H_n\}$-free.

A \emph{$q$-coloring} of a graph $G$ is an assignment of colors from a set $\{1,2,\ldots,q\}$ to each vertex of $G$ such that adjacent vertices receive different colors. We say that a graph $G$ is \emph{$q$-colorable} if $G$ admits a $q$-coloring. In other words, a graph is $q$-colorable if its vertex-set can be partitioned into $q$ stable sets. The \emph{chromatic number} of a graph $G$, denoted by $\chi(G)$, is the minimum number $q$ for which $G$ is $q$-colorable. Clearly,  a graph $G$ is $q$-colorable if and only if $\chi(G)\leq q$.

  The \emph{clique number} of $G$, denoted by $\omega(G)$, is the size of a largest clique in $G$. A graph $G$ is  $K_t$-free if and only if $\omega(G)\leq t-1$. Clearly,  every graph $G$ satisfies $\chi(G)\geq \omega(G)$. A graph $G$ is \emph{perfect} if $\chi(H)=\omega(H)$ for each induced subgraph $H$ of $G$, and \emph{imperfect} otherwise.

A graph class $\mathcal{G}$ is \emph{hereditary} if $G\in\mathcal{G}$ implies that every induced subgraph of $G$ belongs to $\mathcal{G}$. Obviously, $\mathcal{G}$ is hereditary if and only if $\mathcal{G}$ is the class of $\mathcal{H}$-free graphs for some $\mathcal{H}$. A hereditary graph class $\mathcal{G}$ is \emph{$\chi$-bounded} if there is a function $f:\mathbb{N}\rightarrow \mathbb{N}$ (with $f(1)=1$ and $f(x)\geq x$, for all $x\in \mathbb{N}$)  such that $\chi(G)\leq f(\omega(G))$ for every $G\in\mathcal{G}$. The function $f$ is called a \emph{$\chi$-binding function} for $\mathcal{G}$. Gy\'arf\'as~\cite{Gy87} conjectured that,  if $H$ is any fixed tree, then the class of $H$-free graphs is $\chi$-bounded, and proved that every $P_t$-free graph $G$ satisfies $\chi(G)\leq (t-1)^{\omega(G)-1}$. Since this $\chi$-binding function is exponential in $\omega$, researchers tried to find a linear $\chi$-binding function for the class of $P_t$-free graphs. Unfortunately, the class of  $P_t$-free graphs has no linear $\chi$-binding function when $t\geq 5$~\cite{FGMT95}. However, some linear $\chi$-binding functions for subclasses of $P_t$-free graphs have been found. For example, every $P_4$-free graph $G$ satisfies $\chi(G)=\omega(G)$~\cite{Seinsche}.  If a graph $G$ is   ($P_6, C_4$)-free or ($P_5, P_4\vee K_1$)-free or  ($P_4+K_1, P_4\vee K_1$)-free, then $G$ satisfies $\chi(G)\leq\lceil\frac{5\omega(G)}{4}\rceil$; see \cite{CKMM18, KMa182,KMa18}. Every ($P_6$,~diamond)-free graph $G$ satisfies $\chi(G)\leq \max\{6, \omega(G)\}$~\cite{GHJM}.

We say that a function $f$ is the \emph{optimal} $\chi$-binding function for a $\chi$-bounded class of graphs $\mathcal{G}$ if for every positive integer $t$, there is a graph $G\in \mathcal{G}$ such that $\omega(G)=t$ and $\chi(G)=f(t)$. For example, $f(x)=x$ is the optimal $\chi$-binding function for the class of perfect graphs. Likewise, $f(x)=\lceil\frac{5x}{4}\rceil$ is the optimal $\chi$-binding function for the class of ($P_6, C_4$)-free graphs,  ($P_5, P_4\vee K_1$)-free graphs and ($P_4+K_1, P_4\vee K_1$)-free graphs, since certain graph in such a class can be constructed from a $C_5$ by `blowing up' each vertex into a clique; see \cite{KMa182}. However, in general,   the problem of finding  the optimal $\chi$-binding function  for a given $\chi$-bounded class of graphs  often seems to be difficult. So it is interesting to study such graphs with small clique number in order to obtain the optimal $\chi$-binding function.
It is worth  mention here that the problem of finding the chromatic number of graphs with small clique number have been motivating
research in graph
theory  and has a long history  starting from the `Four color problem'. Recent surveys and literature show that it still receives special attention.       A constructive proof of Mycielski \cite{Mycielski} says  that for any positive integer $k$, there is a graph with chromatic number $k$ and clique number 2.
  We give below few examples  for the  chromatic number of graphs (in particular, graphs defined by some forbidden induced subgraphs) with small clique number. We refer to the recent surveys \cite{RS-survey, ScottSey-Survey} for more details and references.

 \smallskip
\no{\it Chromatic number of graphs with clique number $2$}:   Let $G$ be a graph with $\omega(G)= 2$.
  \begin{itemize}\itemsep=0pt
  \item Brandt \cite{Brandt} showed that, if $G$ is $3K_2$-free, then $G$ is $4$-colorable.
  \item Broersma et al. \cite{BGPS}  showed that, if $G$ is ($P_2+ P_4$)-free, then $G$ is $4$-colorable.
    \item Pyatkin \cite{Pyatkin} showed that, if $G$ is $2P_3$-free, then $G$ is $4$-colorable.
    \item Fan et al.~\cite{FXYY} showed that, if $G$ is fork-free with odd-girth at least $7$, then $G$ is $3$-colorable.
\end{itemize}

 \smallskip
\no{\it Chromatic number of graphs with clique number at most $3$}:   Let $G$ be a graph with $\omega(G)\leq 3$.
  \begin{itemize}\itemsep=0pt
  \item Gaspers and the first author~\cite{GH19} showed that, if $G$ is $2K_2$-free, then $G$ is $4$-colorable.
  \item Gravier, Ho\'ang and Maffray~\cite{GHM} showed that, if  $G$ is  $P_t$-free ($t\geq 4$), then $G$ is $(t -2)^{2}$-colorable.
 \item Esperet et al. \cite{ELMM} showed that, if $G$ is $P_5$-free, then $G$ is $5$-colorable.
   \item Chudnovsky et al. \cite{CSRT-K4-free} showed that, if $G$ is odd-hole-free, then $G$ is $4$-colorable.
  \end{itemize}
In this paper, we are interested in finding the optimal $\chi$-binding function for the class of ($P_5$,~kite)-free graphs.  Here, the \emph{kite} is the graph obtained from a $P_4$ by adding a vertex and making it adjacent to all vertices in the $P_4$ except one vertex with degree 1.
Recently, Brause and Gei{\ss}er~\cite{BG21} showed that every ($P_5$, kite)-free graph $G$ satisfies $\chi(G)\leq 3$ (if $\omega(G)\leq 2$) and $\chi(G)\leq 2\omega(G)-2$ (if $\omega(G)\geq 3$). However, the function $f(x)=2x-2$ for $x\geq 3$ does not seem to be optimal for such class of graphs, and we have the following.

\begin{problem}
What is the optimal $\chi$-binding function for the class of ($P_5$,~kite)-free graphs?
\end{problem}
The above problem seems to be difficult in general, and is open even for a subclass of the class of ($P_5$,~kite)-free graphs, namely the class of ($2K_2$, $K_3+K_1$)-free graphs. So we focussed our attention on the class of ($P_5$,~kite)-free graphs with small clique number.
It follows easily from  the earlier mentioned result of Brause and Gei{\ss}er~\cite{BG21} that   every ($P_5$,~kite, $K_3$)-free graph is $3$-colorable, every ($P_5$,~kite, $K_4$)-free graph is $4$-colorable, and that every ($P_5$,~kite, $K_5$)-free graph is $6$-colorable. 
In this paper, we proved the following.


\begin{theorem}\label{thm:p5k-w5}
Every ($P_5$,~kite, $K_6$)-free graph   is $7$-colorable.
\end{theorem}

\begin{theorem}\label{thm:p5k-w6}
Every ($P_5$,~kite, $K_7$)-free graph   is $9$-colorable.
\end{theorem}
The proofs of \cref{thm:p5k-w5,thm:p5k-w6} are given in \cref{sec:proof}.  Clearly, \cref{thm:p5k-w5,thm:p5k-w6}  together with aforementioned known results immediately imply the following
theorem.

\begin{theorem}\label{thm:omega6}
Every ($P_5$,~kite)-free graph $G$ with $\omega(G)\leq 6$ satisfies $\chi(G)\leq\lfloor\frac{3\omega(G)}{2}\rfloor$.
\end{theorem}

In view of \cref{thm:omega6},  we have the following.

\begin{conjecture}\label{con:bound}
Every ($P_5$, kite)-free graph $G$ satisfies $\chi(G)\leq \lfloor\frac{3\omega(G)}{2}\rfloor$.
\end{conjecture}
If the conjecture is true, then the function $f(x)=\lfloor\frac{3x}{2}\rfloor$ is the optimal $\chi$-binding function for the class of ($P_5$,~kite)-free graphs $\cal G$, since given $n\in \mathbb N$, for $\omega(G)=2n$, we have $G=\overline{nC_5}$, and for $\omega(G)=2n+1$, we have $G=\overline{nC_5}\vee K_1$ so that $G\in \cal G$ and that $\chi(G)=\lfloor\frac{3\omega(G)}{2}\rfloor$.  We also note that the conjecture is true for a subclass of the class of ($P_5$, kite)-free graphs, namely the class of ($3K_1$, kite)-free graphs~\cite{CKS-2}.

\smallskip
Moreover, as a corollary of \cref{thm:omega6}, we have the following,    which is a slight improvement to the bound given by Brause and Gei{\ss}er~\cite{BG21}.

\begin{corolarry}\label{thm:2omega-3}
Every  ($P_5$,~kite)-free graph $G$ with $\omega(G)\geq 5$ satisfies $\chi(G)\leq 2\omega(G)-3$.
\end{corolarry}

We finish this section with some more notation and preliminaries used in this paper.   For a positive integer $k$, if an index (say) $i\in \{1,2,\ldots, k\}$ with $i$ modulo $k$, then we simply write $i\in \langle k \rangle$.
Let $G$ be a graph with vertex-set $V(G)$ and edge-set $E(G)$. We say a vertex $u$ is a \emph{neighbor} (\emph{nonneighbor}) of another vertex $v$ in $G$, if $u$ and $v$ are adjacent (nonadjacent). The set of neighbors of a vertex $v$ in $G$ is denoted by $N_G(v)$, and we often write $N(v)$ if the context is clear, and we write $\overline{N}(v)$ to denote the set $V(G)\sm (N(v)\cup \{v\})$.
 For $S\subseteq V(G)$, let $G[S]$ denote the subgraph of $G$ induced by $S$, and we say that $S$ induces a graph $H$ if $G[S]$ is isomorphic to $H$.

 Let $S$ and $T$ be  any two subsets
of $V$. We write $[S,T]$ to denote the set of edges that has
one end in $S$ and other end in $T$. We say that $S$ is
\emph{complete} to $T$ or $[S,T]$ is complete if every vertex in $T$
is adjacent to every vertex in $T$; and $S$ is \emph{anticomplete} to
$T$ if $[S,T]=\emptyset$.  If $S$ is singleton, say $\{v\}$, we simply
write $v$ is complete (anticomplete) to $T$ instead of writing $\{v\}$
is complete (anticomplete) to $T$.
We say that $X$ \emph{dominates} $Y$ if every vertex in $Y$ has a neighbor in $X$,  and we say that a graph $H$ is {\em dominating} in $G$ if $V(H)$ dominates $V(G)\sm V(H)$.

\medskip We also need the following  from the literature to prove our  results.

\begin{rlt}[\cite{GH19}]\label{thm:2k2-k4}
Every ($2K_2$, $K_4$)-free graph is 4-colorable.
\end{rlt}

The structure of perfect graphs is characterized by the following Strong Perfect Graph Theorem~\cite{CRST06}:

\begin{rlt}[\cite{CRST06}]\label{thm:SPGT}
A graph is perfect if and only if it does not contain an odd hole or an odd antihole.
\end{rlt}

The $\chi$-binding function for ($P_5$, kite)-free graphs can be reduced to the $\chi$-binding function for ($2K_2$, $K_3$+$K_1$, $C_5$+$K_1$)-free graphs~\cite{BG21}:

\begin{rlt}[\cite{BG21}]\label{thm:binding}
If a function $f$ is a $\chi$-binding function for the class of ($2K_2$, $K_3+K_1$, $C_5+K_1$)-free graphs, then it is also a $\chi$-binding function for the class of ($P_5$, kite)-free graphs.
\end{rlt}

Next,
 we observe some simple properties of a  ($2K_2$, $K_3+K_1$, $C_5+K_1$, $K_t$)-free graph $G$ where $t\geq 4$, and are given below:
\begin{enumerate}[label=(O\arabic*)]\itemsep=0pt
\item\label{nei-Kfree} For any $v\in V(G)$, $G[N(v)]$ is $K_{t-1}$-free; otherwise $G[N(v)\cup \{v\}]$ contains a $K_t$.
\item\label{nonnei-bip} For any $v\in V(G)$, since $G[\overline{N}(v)]$ is ($2K_2, K_3, C_5$)-free, $G[\overline{N}(v)\cup \{v\}]$ is a bipartite graph.
\item\label{dom-k3} Since $G$ is $K_3+K_1$-free, every triangle in $G$ is dominating.
\item\label{dom-C5} If $G$ is $C_5+K_1$-free, then  every $C_5$ in $G$ is dominating.
 \end{enumerate}

\section{Proofs of our results}\label{sec:proof}

In this section, we prove  \cref{thm:p5k-w5,thm:p5k-w6}, and \cref{thm:2omega-3}.  We first prove the following.

\begin{lemma}\label{lem:omega+2}
Let $\mathcal{G}$ be a hereditary graph family such that for every graph $G\in\mathcal{G}$ and for every vertex $v$ in $G$, $\overline{N}(v)$ induces a bipartite graph. Let $m$ be a positive integer. If every graph in $\mathcal{G}$ with clique number at most $m$ is $q$-colorable, then every graph in $\mathcal{G}$ with clique number at most $m+1$ is ($q+2$)-colorable.
\end{lemma}
\begin{proof}
Let $G\in\mathcal{G}$ be such that the clique number of $G$ is at most $m+1$. Let $v$ be an arbitrary vertex in $G$. Then $G[N(v)]\in\mathcal{G}$ has clique number at most $m$ and therefore is $q$-colorable, while $G[\{v\}\cup\overline N(v)]$ is a bipartite graph and therefore is 2-colorable. So $G$ is ($q+2$)-colorable.
\end{proof}

\medskip

To prove \cref{thm:p5k-w5,thm:p5k-w6,thm:2omega-3}, by \cref{thm:binding}, \ref{nei-Kfree} , \ref{nonnei-bip},  and \cref{lem:omega+2}, it is enough to prove the following theorem.

\begin{theorem}\label{thm:main}
Every ($2K_2$, $K_3+K_1$, $C_5+K_1$, $K_6$)-free graph is 7-colorable.
\end{theorem}
\begin{proof}
Let $G$ be a ($2K_2$, $K_3+K_1$, $C_5+K_1$, $K_6$)-free graph.  We consider two cases depending on whether the graph $G$ contains a $C_5$, and then the proof follows from \cref{thm:c5-free} and \cref{thm:ghasc5} given below.
\end{proof}

\begin{lemma}\label{thm:c5-free}
Every ($2K_2$, $K_3+K_1$, $C_5$, $K_6$)-free graph is $7$-colorable.
\end{lemma}

\begin{lemma}\label{thm:ghasc5}
Every ($2K_2$, $K_3+K_1$, $C_5+K_1$, $K_6$)-free graph that contains a $C_5$ is $7$-colorable.
\end{lemma}

 In the next two subsections, we will prove these two lemmas.

\subsection{Proof of Lemma \ref{thm:c5-free}}\label{ssc:c5-free}

By Lemma \ref{lem:omega+2}, we may reduce Lemma \ref{thm:c5-free} to the following lemma.

\begin{lemma}\label{thm:c5k5-free}
Every ($2K_2$, $K_3+K_1$, $C_5$, $K_5$)-free graph is 5-colorable.
\end{lemma}
\begin{proof}
  We consider two cases depending on whether the graph  contains a $\overline{C_9}$, and  the proof follows from  \cref{thm:c9bar,thm:c7bar} given below.
\end{proof}

\begin{lemma}\label{thm:c9bar}
Every ($2K_2$, $K_3+K_1$, $C_5$, $K_5$)-free graph which contains a  $\overline{C_9}$ is 5-colorable.
\end{lemma}
\begin{proof}
Let $G:=(V,E)$ be a ($2K_2$, $K_3+K_1$, $C_5$, $K_5$)-free graph, and let $Q:=\{v_1,v_2,\ldots,v_9\}$ induces a $\overline{C_9}$ in $G$ such that $v_iv_{i+1}\notin E$ for $i\in \langle 9 \rangle$.  For each $i\in \langle 9 \rangle$, let:
\begin{eqnarray*}
A_i &:=& \{u\in V\setminus Q\mid N_Q(u)=Q\setminus \{v_{i+3},v_{i+4},v_{i+5},v_{i+6}\}\}, \mbox{ and}\\
B_i&:=&\{u\in V\setminus Q\mid N_Q(u)=Q\setminus \{v_{i-1},v_{i},v_{i+1}\}\}.
\end{eqnarray*}
Let $A:=\cup_{i=1}^9A_i$ and $B:=\cup_{i=1}^9B_i$. Then the following properties hold, where   $i\in \langle 9 \rangle$.
\begin{enumerate}[label=(\arabic*),series=edu*] \setlength{\itemsep}{0pt}
\item\label{c9bar:V-Q:prop} For any $u\in V\sm Q$, the following hold:

\vspace{-0.25cm}
\begin{enumerate}[label=(\roman*)]
\item $N(u)\cap \{v_i,v_{i+2},v_{i-2}\}\neq\emptyset$. (Otherwise, \{$u,v_i,v_{i+2},v_{i-2}$\} induces a $K_3+K_1$.)

\item If $uv_i\in E$, then either $uv_{i+1}\in E$ or $uv_{i-1}\in E$.  (Otherwise, \{$u,v_i,v_{i-1},v_{i+1}$\} induces a $2K_2$.)

\item If $uv_i,uv_{i+1}\in E$, then either $uv_{i+2}\in E$ or $uv_{i-1}\in E$. (Otherwise, \{$u,v_i,v_{i+2},v_{i-1},$ $v_{i+1}$\} induces a $C_5$.)

\item If $uv_i\notin E$, then either $uv_{i+1}\notin E$ or $uv_{i-1}\notin E$. (Otherwise, \{$v_i,u,v_{i+1},v_{i-1}$\} induces a $K_3+K_1$.)

\item $u$ has a non-neighbor in $\{v_i,v_{i+2},v_{i+4},v_{i+6}\}$, and in $\{v_{i+1},v_{i+3},v_{i+5},v_{i-1}\}$. (Otherwise, \{$v_i,v_{i+2},v_{i+4},v_{i+6},u$\} (or  \{$v_{i+1},v_{i+3},v_{i+5},v_{i-1},u$\})  induces a $K_5$.)
\end{enumerate}

\item\label{clm:c9bar-part} $V=Q\cup A\cup B$.

\begin{proof2}
Suppose that there is a vertex $u\in V\setminus (Q\cup A\cup B)$. Then since $G$ is $K_3+K_1$-free, every vertex in $V\setminus Q$ has a neighbor in $Q$.
Also, note that, by \ref{c9bar:V-Q:prop}:(ii), (iii) and (iv), $u$ is complete to three consecutive vertices of $Q$. Now we obtain a contradiction in three cases as follows.

\vspace{-0.25cm}
\begin{itemize}\itemsep=0pt
\item Suppose that there is an index $i\in \langle 9 \rangle$ such that $u$ is complete to $\{v_i,v_{i+1},\ldots,v_{i+4}\}$, say $i=1$. Then by \ref{c9bar:V-Q:prop}:(v), $uv_7,uv_8\notin E$. Also, if  both $uv_9\in E$ and $uv_6\in E$, then we have a contradiction to \ref{c9bar:V-Q:prop}:(v). But then $u\in A\cup B$, a contradiction.

\item Suppose that $u$ is not complete to any five consecutive vertices of $Q$.  Now suppose that there is an index $i\in \langle 9 \rangle$ such that $u$ is complete to $\{v_i,v_{i+1},v_{i+2},v_{i+3}\}$, say $i=1$. Then $uv_5,uv_9\notin E$.
So, by \ref{c9bar:V-Q:prop}:(iv), $uv_6,uv_8\notin E$, and by \ref{c9bar:V-Q:prop}:(i), $uv_7\in E$, a contradiction to \ref{c9bar:V-Q:prop}:(ii).
\item Suppose that $u$ is not complete to any four consecutive vertices of $Q$. We may assume that $u$ is complete to $\{v_1,v_2,v_3\}$. Then $uv_4,uv_9\notin E$. Then by \ref{c9bar:V-Q:prop}:(iv), $uv_5,uv_8\notin E$, and hence again by \ref{c9bar:V-Q:prop}:(iv), $uv_6,uv_7\in E$ which is a contradiction to   \ref{c9bar:V-Q:prop}:(iii).
\end{itemize}
The above contradictions prove that \ref{clm:c9bar-part} holds.
\end{proof2}
\end{enumerate}
Next we prove some useful properties of the  subsets $A$ and $B$.

\begin{enumerate} [label=(\arabic*),resume=edu*]
\setlength{\itemsep}{0pt}

\item\label{clm:c9bar-aa} $A_i$ is a stable set, and $[A_i, A_{i+1}\cup A_{i-1}]$ is  anti-complete.

\begin{proof2}
Suppose there are adjacent vertices, say, $a$ and $a'$, in one of the stated lists. If $a,a'\in A_i$, then \{$a,a',v_{i+3},v_{i-3}$\} induces a $2K_2$. If $a\in A_i$ and $a'\in A_{i+1}$, then \{$a,a',v_{i-3},v_{i+4}$\} induces a $2K_2$. 
\end{proof2}

\item\label{clm:c9bar-bb} $B_i$ is a stable set, and  $[B_i, B_{i+1}\cup B_{i-1}]$ is anti-complete.

\begin{proof2}
Suppose there are adjacent vertices, say, $b$ and $b'$, in one of the stated lists. If $b,b'\in B_i$, then $\{b,b',v_{i+1},v_{i-1}\}$ induces a $2K_2$.  If $b\in B_i$ and $b'\in B_{i+1}$, then $\{b,b',v_{i+3},v_{i-4},v_{i-2}\}$ induces a $K_5$. 
\end{proof2}

\item\label{clm:c9bar-ab} $[A_i,B_{i+4}\cup B_{i-4}]$ is anti-complete.

\begin{proof2}
If there are adjacent vertices, say $a\in A_i$ and $b\in B_{i+4}$, then \{$a,b,v_{i-2},v_i,v_{i+2}$\} induces a $K_5$.
\end{proof2}

\item\label{clm:c9bar-ab+3} For any vertex $a\in A_i$, one of $N(a)\cap B_{i+3}$ and $N(a)\cap B_{i-3}$ is empty.

\begin{proof2}
Suppose that there are vertices, say, $b\in N(a)\cap B_{i+3}$ and $b'\in N(a)\cap B_{i-3}$. Then since $\{b,v_{i+5},b',v_{i+4}\}$ does not induce a $2K_2$, we have $b_1b_2\in E$. But then \{$a,b,b',v_{i-1},$ $v_{i+1}$\} induces a $K_5$.
\end{proof2}
\end{enumerate}
 First suppose that  $[A_i, B_{i+3}]$ or $[A_i, B_{i-3}]$ is anti-complete for some $i\in \langle 9 \rangle$; say $[A_9,B_3]$ is anti-complete.  Then we define the following sets: $S_1:= \{v_1,v_2\}\cup A_6\cup A_7\cup B_2$,  $S_2 :=\{v_3,v_4\}\cup A_8\cup A_9\cup B_3\cup B_4$,
 \begin{eqnarray*}
 S_3 &:= & \{v_5,v_6\}\cup A_1\cup B_5\cup B_6 \cup \{a\in A_2\mid N(a)\cap B_5= \emptyset\},\\
  S_4 &:=& \{v_9\}\cup A_5\cup B_1\cup B_9 \cup \{a\in A_4\mid N(a)\cap B_1= \emptyset\}, \mbox{ and }\\
S_5 &:=& \{v_7,v_8\}\cup A_3\cup B_7\cup B_8\cup \{a\in A_2\mid N(a)\cap B_5\neq \emptyset\}\cup \{a\in A_4\mid N(a)\cap B_1\neq \emptyset\}.
\end{eqnarray*}
Then clearly $V = \cup_{j=1}^5S_j$. Moreover,  we have the following:
\begin{claim}\label{col-c9bar-case1}
 $S_1, S_2, \ldots, S_5$ are stable sets.
\end{claim}
\begin{proof2}
By \ref{clm:c9bar-aa}, \ref{clm:c9bar-bb}  and \ref{clm:c9bar-ab}, and by the definitions of $A$ and $B$, clearly $S_1,\ldots, S_4$, and $\{v_7,v_8\}\cup A_3\cup B_7\cup B_8$ are stable sets. Now if some $x\in A_2$ has a neighbor in $B_5$, then, by \ref{clm:c9bar-ab+3}, $N(x)\cap B_8=\es$, and so by the above properties $\{v_7,v_8\}\cup A_3\cup B_7\cup B_8 \cup \{a\in A_2\mid N(a)\cap B_5\neq \emptyset\}$ is a stable set. Likewise, $\{v_7,v_8\}\cup A_3\cup B_7\cup B_8\cup \{a\in A_4\mid N(a)\cap B_1\neq \emptyset\}$ is a stable set. So, by \ref{clm:c9bar-aa}, we conclude that $S_5$ is also a stable set. This proves \cref{col-c9bar-case1}.
\end{proof2}

\smallskip
By \cref{col-c9bar-case1}, we conclude that $\chi(G)\leq 5$, hence $G$ is $5$-colorable.

So we may assume that for each $i\in \langle 9 \rangle$, both $[A_i, B_{i+3}]$ and $[A_i,B_{i-3}]$ are not anti-complete.
Then we have the following.
\begin{claim}\label{clm:c9bar-ba+3}
 For any vertex $b\in B_i$, one of $N(b)\cap A_{i+3}$ and $N(b)\cap A_{i-3}$ is empty.
  \end{claim}
\begin{proof2}
We will prove for $i=1$. Suppose that there are vertices, say, $a\in N(b)\cap A_4$ and $a'\in N(b)\cap A_7$. Then since \{$b,a,a',v_1$\} does not induce a $K_3+K_1$, we have $aa'\notin E$.  By our assumption, there are vertices $x\in A_6$ and  $y\in B_9$ such that $xy\in E$. Now, by \ref{clm:c9bar-ab}, $xb,ay\notin E$, and, by \ref{clm:c9bar-bb}, $by\notin E$. Then since  \{$a,b,x,y$\} does not induce a $2K_2$, we have $ax\in E$, and since \{$a',b,x,y$\} does not induce a $2K_2$, we have $a'y\in E$. But then $\{b,a,x,y,a'\}$ induces a $C_5$. This proves \cref{clm:c9bar-ba+3}.
\end{proof2}

\smallskip
Now, by  \cref{clm:c9bar-ba+3}, we   define the following sets:
\begin{eqnarray*}
 S_1 &:= & \{v_1,v_2\}\cup A_6\cup B_1\cup B_2 \cup \{a\in A_7\mid N(a)\cap B_1= \emptyset\},\\
 S_2 &:= & \{v_3,v_4\}\cup A_8\cup B_3\cup B_4 \cup \{a\in A_7\mid N(a)\cap B_1\neq \emptyset\}\cup \{a\in A_9\mid N(a)\cap B_6\neq \emptyset\} ,\\
 S_3 &:= & \{v_5,v_6\}\cup A_1\cup B_5\cup B_6 \cup  \{a\in A_9\mid N(a)\cap B_6= \emptyset\},\\
  S_4 &:=& \{v_7,v_8\}\cup A_2\cup A_3\cup B_7 \cup \{b\in B_8\mid N(b)\cap A_2= \emptyset\}, \mbox{ and }\\
S_5 &:=& \{v_9\}\cup A_4\cup A_5\cup B_9\cup \{b\in B_8\mid N(b)\cap A_2\neq \emptyset\}.
\end{eqnarray*}
Then clearly $V = \cup_{j=1}^5S_j$. Moreover, we have the following:
\begin{claim}\label{col-c9bar-case2}
 $S_1, S_2, \ldots, S_5$ are stable sets.
\end{claim}
\begin{proof2}
By the definitions of $A$ and $B$, by \ref{clm:c9bar-aa},~\ref{clm:c9bar-bb} and \ref{clm:c9bar-ab}, and by \cref{clm:c9bar-ba+3}, it follows that $S_1,S_2,\ldots,S_5$ are stable sets. Since the proof is  similar to that of \cref{col-c9bar-case1},   we omit the details.
\end{proof2}

\medskip
By \cref{col-c9bar-case2}, we conclude that $\chi(G)\leq 5$, and hence $G$ is $5$-colorable.
This proves   \cref{thm:c9bar}.
\end{proof}

\begin{lemma}\label{thm:c7bar}
Every ($2K_2$, $K_3+K_1$, $C_5$, $K_5$, $\overline{C_9}$)-free graph is 5-colorable.
\end{lemma}
\begin{proof}
Let $G=(V,E)$ be a ($2K_2$, $K_3+K_1$, $C_5$, $K_5$, $\overline{C_9}$)-free graph. First suppose that $G$ is ($\overline{C_7}\vee K_1$)-free. Let $v\in V$. Then since $G$ is $K_5$-free, by \cref{thm:SPGT}, $G[N(v)]$ is a $K_4$-free perfect graph and therefore is $3$-colorable, and $G[\overline{N}(v)\cup\{v\}]$ is a bipartite graph and therefore is $2$-colorable. Hence $G$ is $5$-colorable. So we may assume that $G$ contains a  $\overline{C_7}\vee K_1$. Let $Q:=\{v_1,v_2,\ldots,v_7\}$ induces a $\overline{C_7}$ in $G$ such that $v_iv_{i+1}\notin E$ for $i\in \langle 7 \rangle$ and let $w$ be a vertex which is complete to $Q$. For each $i\in \langle 7 \rangle$, let:
\begin{eqnarray*}
 A_i&:=&\{u\in V\setminus Q \mid N(u)\cap Q=\{v_{i-1},v_i,v_{i+1}\}\},\\
B_i&:=&\{u\in V\setminus Q \mid N(u)\cap Q= Q\sm \{v_{i-1},v_i,v_{i+1}\}\},\\
D_i&:=&\{u\in V\setminus Q \mid N(u)\cap Q= Q\sm\{v_{i-3},v_{i+3}\}\}, \mbox{ and}\\
W&:=&\{u\in V\setminus Q \mid N(u)\cap Q=Q\}.
\end{eqnarray*}
Let $A:=\bigcup_{i=1}^7A_i$, $B:=\bigcup_{i=1}^7B_i$ and $D:=\bigcup_{i=1}^7D_i$. Clearly $w\in W$. Further:
\begin{enumerate}[label=(\arabic*),series=edu*]
\item\label{c7bar-part}
$V=Q\cup A\cup B\cup D\cup W$.

\begin{proof2}
We observe that  first four of the five properties listed in   Lemma~\ref{thm:c9bar}:\ref{c9bar:V-Q:prop} still hold. Using those properties, it is easily verified that $V=Q\cup A\cup B\cup D\cup W$. The proof is similar to that of Lemma~\ref{thm:c9bar}:\ref{clm:c9bar-part}, and we omit the details.
\end{proof2}
\end{enumerate}
Moreover, the following  properties hold, where $i\in \langle 7 \rangle$.

\begin{enumerate}[label=(\arabic*),resume=edu*]

\item\label{clm:c7bar-aa} $A_i$ is a stable set, and $[A_i, A_{i+1}\cup A_{i-1}]$ is  anti-complete.

\begin{proof2}
Suppose there are adjacent vertices, say, $a$ and $a'$, in one of the listed sets. If $a,a'\in A_i$, then \{$a,a',v_{i-2},v_{i+2}$\} induces a $2K_2$. If $a\in A_i$ and $a'\in A_{i+1}$, then \{$a,a',v_{i+3},v_{i-2}$\} induces a $2K_2$. 
\end{proof2}

\item\label{clm:c7bar-aa+3} One of $A_i$ and $A_{i+3}\cup A_{i-3}$ is empty.

\begin{proof2}
Suppose that there are vertices, say $a\in A_i$ and $a'\in A_{i+3}$. If $aa'\in E$, then \{$a,v_{i+1},v_{i-2},$ $v_{i+2},a'$\} induces a $C_5$. If $aa'\notin E$, then \{$a,v_{i+1},a',v_{i+2}$\} induces a $2K_2$.
\end{proof2}

\item\label{clm:c7bar-bb} $B_i$ is a stable set, and  $[B_i, B_{i+1}\cup B_{i-1}]$ is anti-complete.

\begin{proof2}
Suppose that there are adjacent vertices, say, $b$ and $b'$, in one of the listed sets. If $b,b'\in B_i$, then  \{$b,b',v_{i-1},v_{i+1}$\} induces a $2K_2$.  If $b\in B_i$ and $b'\in B_{i+1}$, then  \{$b,b',v_{i-2},v_{i+3},w$\} induces a $K_5$. 
\end{proof2}

\item\label{clm:c7bar-dx} $D_i$ and $W$ are stable sets.

\begin{proof2}
Suppose there are adjacent vertices, say, $x$ and $y$, in one of the listed sets.
      If $x,y\in D_i$, then \{$x,y,v_{i-2},v_i,v_{i+2}$\} induces a $K_5$. If $x,y\in W$, then \{$x,y,v_1,v_3,v_5$\} induces a $K_5$.
\end{proof2}

\item\label{clm:c7bar-d} $[W, A\cup B]$ is complete, and $[W, D]$ is anti-complete. Moreover, $D$ induces a bipartite graph.

\begin{proof2}
If there are nonadjacent vertices, say $x\in W$ and $a\in A_i$, then \{$x,v_{i-2},v_{i+2},a$\} induces a $K_3+K_1$. If there are nonadjacent vertices, say $x\in W$ and $b\in B_i$, then \{$x,v_{i-1},v_{i+1},b$\} induces a $K_3+K_1$. If there are adjacent vertices, say $x\in W$ and $d\in D_i$, then \{$x,d,v_{i-2},v_i,v_{i+2}$\} induces a $K_5$. This proves the first assertion of this claim. Since $W$ is anti-complete to $D$, we see that $D\subseteq\overline N(w)$ and therefore $G[D]$ is a bipartite graph.
\end{proof2}

\item\label{clm:c7bar-ab} $[A_i, B_{i+3}\cup B_{i-3}]$ is   anti-complete.

\begin{proof2}
 If there are adjacent vertices, say, $a\in A_i$ and $b\in B_{i+3}$, then \{$a,b,v_{i-1},v_{i+1},w$\} induces a $K_5$.
\end{proof2}

\item\label{clm:c7bar-ab+2} For any vertex in $a\in A_i$, one of $N(a)\cap B_{i+2}$ and $N(a)\cap B_{i-2}$ is empty.

\begin{proof2}
Suppose that there are vertices, say, $b\in N(a)\cap B_{i+2}$ and $b'\in N(a)\cap B_{i-2}$. Then since $\{b',v_{i+3},b,v_{i-3}\}$ does not induce a $2K_2$, we have $bb'\in E$. But then, by \ref{clm:c7bar-d}, \{$a,b,b',v_i,w$\} induces a $K_5$.
\end{proof2}

\item\label{clm:c7bar-cc} $[D_i, D_{i+2}\cup D_{i-2}\cup D_{i+3}\cup D_{i-3}]$ is complete, and $[D_i, D_{i+1}\cup D_{i-1}]$ is anti-complete.

\begin{proof2}
If there are nonadjacent vertices, say, $x\in D_i$ and $y\in D_{i+2}$, then \{$x,v_{i-2},y,v_{i-3}$\} induces a $2K_2$. If there are nonadjacent vertices, say, $x\in D_i$ and $y\in D_{i+3}$, then \{$x,v_{i-2},y,v_{i-3},v_{i-1}$\} induces a $C_5$. Finally, if there are adjacent vertices, say, $x\in D_i$ and $y\in D_{i+1}$, then, by \ref{clm:c7bar-d},  \{$x,y,w,v_{i-3}$\} induces a $2K_2$.
\end{proof2}

\item\label{clm:c7bar-ac} $[A_i, D_{i+1}\cup D_{i-1}]$ is anti-complete. Moreover if $A_i\neq\emptyset$, then one of $D_{i+1}$ and $D_{i-1}$ is empty.

\begin{proof2}
  Suppose there are adjacent vertices, say $a\in A_i$ and $x\in D_{i+1}$, then \{$a,x,v_{i-1},v_{i-2}$\} induces a $K_3+K_1$. This proves the first assertion of this claim. Suppose that there are vertices, say $a\in A_i$, $x\in D_{i+1}$, and $y\in D_{i-1}$. Then $ax,ay\notin E$, and by \ref{clm:c7bar-cc}, $xy\in E$. But then, by \ref{clm:c7bar-d}, \{$a,w,x,y$\} induces a $2K_2$.
\end{proof2}

\item\label{clm:c7bar-aad} $[A_i, A_{i+2}\cup A_{i-2}\cup D_{i+2}\cup D_{i-2}]$  is complete.

\begin{proof2}
  If there are nonadjacent vertices, say $a\in A_i$ and $a'\in A_{i+2}$, then \{$a,v_{i+1},a',v_{i+2},v_i$\} induces a $C_5$.
   If there are nonadjacent vertices, say $a\in A_i$ and $x\in D_{i+2}$, then \{$a,x,v_{i+2},v_{i-3}$\} induces a $K_3+K_1$.
   \end{proof2}

\item\label{clm:c7bar-bc} $[B_i, D_{i+3}\cup D_{i-3}]$ is anti-complete.

\begin{proof2}
 If there are adjacent vertices, say $b\in B_i$ and $d\in D_{i+3}$ are adjacent, then \{$b,d,v_{i-2},v_{i-1}$\} induces a $K_3+K_1$.
\end{proof2}

\end{enumerate}

By \ref{clm:c7bar-aa+3}, $A=A_i\cup A_{i+1}\cup A_{i+2}$ for some $i\in \langle 7 \rangle$. By symmetry, we may assume that $A=A_2\cup A_3\cup A_4$. Now we give a 5-coloring of $G-D$:
\begin{itemize}\itemsep=0pt
\item Color $\{v_5,v_6\}\cup A_2\cup B_5\cup B_6$ with color 1.

\item Color $\{v_1,v_7\}\cup A_4\cup B_1\cup B_7$ with color 2.

\item Color $\{v_2\}\cup B_2$ with color 3.

\item Color $\{v_4\}\cup B_4$ with color 4.

\item Color $W$ with color 5.

\item For each vertex in $A_3$, color it with color 1 if it has a neighbor in $B_1$, otherwise color it with color 2.

\item Any vertex in $\{v_3\}\cup B_3$ can be colored with color 3 or 4. The specific coloring of $\{v_3\}\cup B_3$ will be given later according to the coloring of $D$.
\end{itemize}
By the properties we proved above, each vertex set with the same color is a stable set. Note that $v_2$ and $v_4$, $v_1$ and $v_5$, $v_6$ and $v_7$, colors 1 and 2, colors 3 and 4, are in symmetric positions.

By \ref{clm:c7bar-d} and \ref{clm:c7bar-cc}, $D$ can be partitioned into two stable sets $S_1$ and $S_2$, such that each of the two sets is equal to $D_i\cup D_{i+1}$ for some $i$.

Now assume that one of $S_1$ and $S_2$ is equal to $D_5\cup D_6$, $D_6\cup D_7$ or $D_7\cup D_1$. By symmetry, we may assume that  $S_1$ is either $D_5\cup D_6$ or $D_6\cup D_7$. If $S_1=D_5\cup D_6$, then color $S_1$ with color 3, $S_2$ with color 5, and $\{v_3\}\cup B_3$ with color 4. If $S_1=D_6\cup D_7$, then color $D_6$ with color 3, $D_7$ with color 4, $S_2$ with color 5, and each vertex in $\{v_3\}\cup B_3$ with any color of 3 or 4. Hence, by above properties, $G$ is $5$-colorable.

So we consider the opposite case: $D$ is $D_2\cup D_3\cup D_4\cup D_5$ or $D_1\cup D_2\cup D_4\cup D_5$ or $D_1\cup D_2\cup D_3\cup D_4$, and $D_2,D_4\neq\emptyset$. By \ref{clm:c7bar-ac}, $A_3=\emptyset$. If $A_2,A_4\neq\emptyset$, by  \ref{clm:c7bar-cc},  \ref{clm:c7bar-ac} and \ref{clm:c7bar-aad}, we see that $A_2\cup A_4\cup D_2\cup D_4\cup\{v_3\}$ contains a $K_5$. So one of $A_2$ and $A_4$, say $A_2$, is empty. If $A_4=\emptyset$, by symmetry we may assume that $S_1=D_1\cup D_2$. If $A_4\neq\emptyset$, by \ref{clm:c7bar-ac}, one of $D_3$ and $D_5$ is empty, then we may assume that $S_1=D_1\cup D_2$. Since $S_1=D_1\cup D_2$, we may color $D_1$ with color 4, $D_2$ with color 1, $S_2$ with color 5, and $\{v_3\}\cup B_3$ with color 3. So, again by above properties, we conclude that $G$ is $5$-colorable. This completes the proof of  \cref{thm:c7bar}.
\end{proof}

\subsection{Proof of Lemma \ref{thm:ghasc5}}\label{ssc:ghasc5}

Let $G=(V,E)$ be a ($2K_2$, $K_3+K_1$, $C_5+K_1$, $K_6$)-free graph, and let $Q:=\{v_1,v_2,\ldots,v_5\}$ induces a $C_5$ in $G$ such that $v_iv_{i+1}\in E$ for $i\in \langle 5 \rangle $. Then there are five non-empty and pairwise disjoint vertex-sets $A_1,A_2,\ldots, A_5$ in $G$ such that, for each $i\in \langle 5 \rangle $, the set $A_i$ is complete to $A_{i-1}\cup A_{i+1}$, and anti-complete to $A_{i-2}\cup A_{i+2}$.  Let $A:=A_1\cup\cdots\cup A_5$.  We choose
these sets so that $A$ is maximal, and let $v_i\in A_i$.  For each $i\in \langle 5 \rangle $, let
 $B_i$ denote the set $\{u\in V\setminus A \mid u $ $\mbox{has a neighbor in each of } A_i, A_{i-2} \mbox{ and } A_{i+2}, \mbox{  and  anti-}$ $\mbox{complete to } A_{i-1}\cup A_{i+1}\}$, and let $D$ denote the set $\{u\in V\setminus A \mid u \mbox{ has a neighbor in } A_i, \mbox{ for all } i\in \langle 5 \rangle \}$.  Let $B:=\bigcup_{i=1}^5B_i$. Then, we observe that the following properties hold, where $i\in \langle 5 \rangle $:

\begin{enumerate}[label=(\arabic*),series=edu]

 \item\label{c5:part} $V=A\cup B\cup D$.

 \begin{proof2}
  Consider any $x\in V\setminus (A\cup B\cup D)$.  For each $i\in \langle 5 \rangle $, let
$a_i$ be a neighbor of $x$ in $A_i$ (if such a vertex exists) and
$b_i$ be a non-neighbor of $x$ in $A_i$ (if such a vertex exists).  Let
$Y:=\{i\mid a_i$ exists$\}$.   Then since $G$ is $C_5+K_1$-free, by \ref{dom-C5}, $x$ has a neighbor in $A$, and hence $Y\neq\emptyset$. Also, since $x\notin D$,  we have $1\leq |Y|\leq 4$. Further, since $x\notin B$, we conclude that $Y\neq \{i-2,i,i+2\}$ for any $i\in \langle 5 \rangle $. Now up to symmetry, $Y$ is one of the following:

\vspace{-0.25cm}
\begin{itemize}\itemsep=0pt
\item $Y=\{i\}$ or $\{i,i+1\}$ or $\{i-1, i,i+1\}$ for some $i\in \langle 5 \rangle $.  Then $\{x,a_i,b_{i-2},b_{i+2}\}$ induces a $2K_2$.
\item $Y=\{i-1,i+1\}$ for some $i\in \langle 5 \rangle $. Then $x$ is complete to $A_{i-1}\cup A_{i+1}$, for otherwise $\{b_{i-1},b_{i-2},x,a_{i+1}\}$ or $\{b_{i+1},b_{i+2},x,a_{i-1}\}$ induces a $2K_2$. So $x$ can be added to $A_i$, contradicting the maximality of $A$.
\item $Y=\{i-1, i,i+1,i+2\}$ for some $i\in \langle 5 \rangle $. Then $\{x,a_{i},a_{i+1},b_{i-2}\}$ induces
$K_3+K_1$.
\end{itemize}

\vspace{-0.25cm}
So \ref{c5:part} holds.
\end{proof2}

\item\label{c5:a-stable} $A_i$ is a stable set.

\begin{proof2}
If there are adjacent vertices, say $a,a'\in A_i$, then $\{a,a',v_{i+2},v_{i-2}\}$ induces a $2K_2$.
\end{proof2}

\item\label{c5:daibi} Any vertex in $B_i\cup D$ is complete to $A_i$. So, $D$ is complete to $A$.

\begin{proof2}
Suppose there are nonadjacent vertices, say $x\in B_i\cup D$ and $a\in A_i$. Then pick a neighbor of $x$ in $A_{i+2}$, say $a'$, and pick a neighbor of $x$ in $A_{i-2}$, say $a''$. Then $\{x,a',a'',a\}$ induces a $K_3+K_1$.
\end{proof2}

\item\label{c5:biai+2ai-2} Any vertex in $B_i$ is complete to either $A_{i+2}$ or $A_{i-2}$.

\begin{proof2}If there is a vertex $b\in B_i$ which has non-neighbors, say $a\in A_{i+2}$ and $a'\in A_{i-2}$, then, by \ref{c5:daibi}, $\{x,v_i,a,a'\}$ induces a $2K_2$.
\end{proof2}

\item\label{clm:c5-dk4} $G[D]$ is $K_4$-free.

\begin{proof2}
If $G[D]$ contains a $K_4$, say $K$, then, by \ref{c5:daibi}, $V(K)\cup\{v_1,v_2\}$ induces a $K_6$.
\end{proof2}

\item\label{clm:c5-bstable}  $B_i$ is a stable set.

\begin{proof2}
Suppose that there are adjacent vertices in $B_i$, say  $b$ and $b'$. By \ref{c5:biai+2ai-2}, we may assume that $b$ is complete to $A_{i+2}$. Now pick a neighbor of $b'$ in $A_{i+2}$, say $a$. Then $ba,b'a\in E$, and then \{$b,b',a,v_{i-1}$\} induces a $K_3+K_1$.
\end{proof2}

\item\label{c5-ab-stable}  $A_{i-1}\cup A_{i+1}\cup B_i$ is a stable set.

\begin{proof2}
This follows from the definitions of $A$ and $B$, and by \ref{c5:a-stable} and \ref{clm:c5-bstable}.
\end{proof2}

\item\label{clm:c5-bb+1} $B_i$ is complete to $B_{i+1}\cup B_{i-1}$.

\begin{proof2}
Suppose that there are nonadjacent vertices, say $b\in B_i$ and $b'\in B_{i+1}$.  Then pick a neighbor of $b$ in $A_{i+2}$, say $a$, and a neighbor of $b'$ in $A_{i-1}$, say $a'$. Now $\{b,a,b',a'\}$ induces a $2K_2$.
\end{proof2}

\item\label{clm:c5-bda+2} Any vertex in $B_i$ is complete to either   $D$ or   $A_{i+2}\cup A_{i-2}$.

\begin{proof2}
Let $b\in B_i$. Suppose that there are vertices, say $d\in D$ and $a\in A_{i+2}$ such that $bd,ba\notin E$. Then \{$b,d,a,v_{i+1}$\} induces a $K_3+K_1$.
\end{proof2}

\item\label{clm:c5-bnd} For any vertex $b\in B_i$,  $D\sm N(b)$ is a stable set.

\begin{proof2}
Suppose that there are vertices, say $d_1,d_2\in D\sm N(b)$ such that $d_1d_2\in E$. Then, $bd_1,bd_2\notin E$, and then \{$b,d_1,d_2,v_{i+1}$\} induces a $K_3+K_1$.
\end{proof2}

\item\label{clm:c5-bdk3} Suppose that $G[D]$ contains a $K_3$, say $K$. Then any vertex in $B_i$ is adjacent to exactly two vertices in $K$.  Moreover, $B_i$ is complete to $A_{i+2}\cup A_{i-2}$.

\begin{proof2} Let $V(K):=\{d_1,d_2,d_3\}$, and let $b\in B_i$.
By \ref{clm:c5-bnd}, $b$ is adjacent to at least two of $d_1,d_2,d_3$. Suppose that $b$ is adjacent to all of $d_1,d_2,d_3$. Now pick a neighbor of $b$ in $A_{i+2}$, say $a$, and in $A_{i-2}$, say  $a'$.  Then \{$b,d_1,d_2,d_3,a,a'$\} induces a $K_6$. This proves the first assertion of this item. The second assertion of this item follows from the first assertion and from \ref{clm:c5-bda+2}.
\end{proof2}

\item\label{c5:dbk3-free} For any vertex $b\in B_i$, $G[N(b)\cap D]$ is  $K_3$-free. So, $G[B_i\cup D]$ is $K_4$-free.

\begin{proof2}
This first assertion follows from \ref{clm:c5-bdk3}.  If $G[B_i\cup D]$ contains a $K_4$, then, by \ref{clm:c5-dk4} and \ref{clm:c5-bstable}, the $K_4$ has three vertices in $D$ and one vertex in $B_i$, which contradicts the first assertion.
\end{proof2}

\item\label{clm:c5-dc7bar} If $B\neq\emptyset$, then $G[D]$ contains no antiholes on 7 or more vertices.

\begin{proof2}
Let $b\in B$. Suppose that $G[D]$ contains an antihole $\overline{C_n}$ ($n\geq 7$) with vertex-set, say   $R:=\{d_1,d_2,\ldots,d_n\}$ and such that $d_id_{i+1}\notin E$ for $i\in \{1,2,\ldots,n\}$ with all indices modulo $n$. By \ref{clm:c5-bnd}, we may assume that $D\sm (N(b)\cap R)\subseteq\{d_1,d_2\}$. Then \{$d_3,d_5,d_7$\} induces a $K_3$ in $N(b)\cap R$ $(\subseteq N(b)\cap D)$, a contradiction to \ref{c5:dbk3-free}.
\end{proof2}

\item\label{clm:c5-bb+2domi} If there are   vertices, say $b\in B_i$ and $b'\in B_{i+2}$ such that $bb'\in E$, then \{$b,b'$\} dominates $D$.

\begin{proof2}
If there is a vertex, say $d\in D$ such that $bd,b'd\notin E$, then \{$b,b',d,v_{i+1}$\} induces a $2K_2$.
\end{proof2}

\item\label{clm:c5-3or4colorable} Let $X\subseteq A\cup B$ such that $X\cap B_i$ and $X\cap B_{i+2}$ are anti-complete for some $i\in \langle 5 \rangle $, then $G[X]$ is 4-colorable. Moreover if $X\cap B_{i+1}=\emptyset$, then $G[X]$ is 3-colorable.

\begin{proof2}
Since $A_{i+1}\cup (X\cap (B_i\cup B_{i+2}))$, $A_{i+2}\cup A_{i+4}\cup B_{i+3}$, $A_i\cup A_{i+3}\cup B_{i+4}$ and $B_{i+1}$ are four stable sets (by \ref{c5-ab-stable}), $G[X]$ is $4$-colorable, and is $3$-colorable if $B_{i+1}=\emptyset$.
\end{proof2}

\end{enumerate}
Now we proceed to prove the lemma.
First if $D=\es$, then, by \ref{c5-ab-stable}, for each $i\in \langle 5 \rangle $, $B_i\cup A_{i+1}$ is a stable set (by \ref{c5-ab-stable}), and hence $G$ is $5$-colorable. So we may assume that $D\neq \es$. Next, suppose that
there is an index $i\in \langle 5 \rangle $ such that $[B_i,B_{i+2}]$ is anti-complete, say $i=1$. Since $G[B_2\cup D]$ is $K_4$-free (by \ref{c5:dbk3-free}) and therefore is 4-colorable, by \cref{thm:2k2-k4}. Also, by \ref{clm:c5-3or4colorable}, $G[A\cup(B\setminus B_2)]$ is 3-colorable. So $\chi(G)\leq \chi(G[A\cup(B\setminus B_2)])+\chi(G[B_2\cup D])\leq 7$, and hence $G$ is $7$-colorable.

So we may assume that $[B_i, B_{i+2}]$ is not anti-complete for each $i\in \langle 5 \rangle $. More precisely, $[B_i, B_{i+2}]\neq \es$.  Then, by \ref{clm:c5-dc7bar}, $G[D]$ contains no antiholes on 7 or more vertices. To prove the lemma, we consider three cases: $G[D]$ contains a  $C_5$, $G[D]$ is $C_5$-free and contains a $\overline{C_6}$, and $G[D]$ is ($C_5$, $\overline{C_6}$)-free, and prove that $G$ is 7-colorable in each case.

\begin{case}\label{C5-case1} $G[D]$ contains a $C_5$.
\end{case}
Let $Q':=\{v'_1,\ldots, v'_5\}$ induces a $C_5$ in $D$, such that $v'_iv'_{i+1}\in E$ for $i\in \langle 5 \rangle $. Then, as earlier, there are five non-empty and pairwise disjoint vertex-sets $A'_1,A'_2,\ldots, A'_5$ in $G$ such that, for each $i\in \langle 5 \rangle $, the set $A'_i$ is complete to $A'_{i-1}\cup A'_{i+1}$, and anti-complete to $A'_{i-2}\cup A'_{i+2}$.  Let $A':=A'_1\cup\cdots\cup A'_5$.  We choose these sets so that $A'$ is maximal, and let $v'_i\in A'_i$.  For each $i\in \langle 5 \rangle $, let
 $B'_i$ denote the set $\{u\in V\setminus A' \mid u $ $\mbox{has a neighbor in each of } A'_i, A'_{i-2} \mbox{ and } A'_{i+2}, \mbox{  and  anti-complete to } A'_{i-1}\cup A'_{i+1}\}$, and let $D'$ denote the set $\{u\in V\setminus A' \mid u \mbox{ has a neighbor in } A'_i, \mbox{ for all } i\in \langle 5 \rangle \}$.  Let $B':=\bigcup_{i=1}^5B'_i$.  Then, as in \ref{c5:part}, $V=A'\cup B'\cup D'$, and  the above properties   hold in $G$ with $A'$, $B'$ and $D'$ in place of $A, B$ and $D$, respectively.
Moreover, the following hold.

\begin{enumerate}[label=(\roman*), series=edu*]

\item\label{clm:c5-c5-a'dad'} $A'\subseteq D$ and $A\subseteq D'$.

\begin{proof2}
Suppose that $u\in A'_i\setminus D$. Then $uv'_{i+2},uv'_{i-2}\notin E$. If $u\in A$, then $u$ is complete to $D$, a contradiction. If $u\in B$, then it contradicts \ref{clm:c5-bnd}. This proves the first part of this claim. Then $A$ is complete to $A'$ and therefore is included in $D'$.
\end{proof2}

\item\label{clm:c5-c5-dd'} $D\cap D'=\emptyset$.

\begin{proof2}
Suppose that  there is a vertex $d\in D\cap D'$. Then for each $i\in \langle 5 \rangle $, $dv_i'\in E$, and so $\{d,v_i',v_{i+1}'\}$ induces a $K_3$. Then, by \ref{clm:c5-bdk3}, for each $b\in B$, $N(b)\cap (Q'\cup\{d\})=Q'$, which contradicts \ref{clm:c5-bb+2domi} and our assumption that $[B_i, B_{i+2}]\neq \es$ for each $i\in \langle 5 \rangle $.
\end{proof2}

\item\label{clm:c5-c5-db'nonempty} If $D\cap B'\neq\emptyset$, then $D'\cap B$ is a stable set.

\begin{proof2}
Suppose that there are vertices, say $d\in D\cap B_i'$ and $b,b'\in D'\cap B$ such that $bb'\in E$. Pick a neighbor  of $d$ in  $A_{i+2}'$, say $a$,  and  a neighbor  of $d$ in $A_{i-2}'$, say $a'$. Then, by \ref{clm:c5-c5-a'dad'}, $a,a'\in D$. Since \{$a,a',d$\} induces a $K_3$ in $D$, by \ref{clm:c5-bdk3}, we have $bd,b'd\notin E$. But, then \{$b,b',v_{i+1}',d$\} induces a $K_3+K_1$.
\end{proof2}

\item\label{clm:c5-c5-bbcompdb} $B\cap B_i'$ and $D\cap(B'\setminus B_i')$ are complete. Moreover, suppose that $B\cap B_i'\neq\emptyset$. If $D\cap B_{i+2}'\neq \emptyset$ (resp. $D\cap B_{i-2}'\neq \emptyset$), then $D\cap B_{i+1}'$ and $D\cap B_{i+3}'$ (resp. $D\cap B_{i-1}'$ and $D\cap B_{i-3}'$) are anti-complete.

\begin{proof2}
Suppose that there are nonadjacent vertices, say $b\in B\cap B_i'$ and $d\in D\cap(B'\setminus B_i')$. By symmetry assume that $d\in D\cap B_{i+1}'$ or $D\cap B_{i+2}'$. Then $d$ has a neighbor $a$ in $A_{i-1}'$, and $a\in D$ by \ref{clm:c5-c5-a'dad'}. Since $ab\notin E$, $\{a,d\}$ induces a $K_2$ in $G[D\setminus N(b)]$, which contradicts \ref{clm:c5-bnd}. This proves the first part of this claim. Suppose that $B\cap B_i'\neq\emptyset$. Then by the first part and \ref{clm:c5-bdk3}, $G[D\cap(B'\setminus B_i')]$ is $K_3$-free. Suppose that there are vertices, say $d_1\in B_{i+1}$, $d_2\in B_{i+2}$ and $d_3\in B_{i+3}$ such that $d_1d_3\in E$. Then $d_1d_2,d_2d_3\in E$ by \ref{clm:c5-bb+1}, and so $\{d_1,d_2,d_3\}$ induces a $K_3$ in $G[D\cap(B'\setminus B_i')]$, a contradiction.
\end{proof2}

\end{enumerate}
To prove the lemma in this case, we consider two cases:

\medskip
\no{\bf Case 1.1}~ {\em $D\cap B'\neq\emptyset$ and $D'\cap B\neq\emptyset$.}

\smallskip
\no Then $D\cap B'$ and $D'\cap B$ are stable sets (by \ref{clm:c5-c5-db'nonempty}). By \ref{clm:c5-bb+1}, $D'\cap B\subseteq B_i\cup B_{i+2}$ for some $i$. Assume that $D'\cap B\subseteq B_1\cup B_3$. Let $X:=D\cup B_4\cup B_5$. Then $X\subseteq A'\cup B'$ by \ref{clm:c5-c5-dd'}. By  \ref{clm:c5-bstable} and \ref{clm:c5-bb+1}, and by symmetry, assume that $B_4\subseteq B_1'\cup B_3'$ and $B_5\subseteq B_2'\cup B_4'$. If $D\cap B_5'=\emptyset$, then $X\cap B_5'=\emptyset$, and then by \ref{clm:c5-3or4colorable}, $G[X]$ is $4$-colorable. If $D\cap B_5'\neq\emptyset$, then, by \ref{clm:c5-bstable} and \ref{clm:c5-bb+1}, $D\cap B'$ is a subset of $B_5'\cup B_2'$ or $B_5'\cup B_3'$. By symmetry let $D\cap B'\subseteq B_5'\cup B_2'$. Then $X\cap (B_1'\cup B_3')\subseteq B_4$, and so $X\cap B_1'$ and $X\cap B_3'$ are anti-complete by \ref{clm:c5-bstable}.  By \ref{clm:c5-3or4colorable}, $G[X]$ is 4-colorable, and $G-X=G[A\cup B_1\cup B_2\cup B_3]$ is $3$-colorable, and hence  $G$ is 7-colorable.

\medskip
\no{\bf Case 1.2}~ {\em One of $D\cap B'$ and $D'\cap B$ is empty.}

\smallskip
\no By symmetry let $D'\cap B=\emptyset$. Then $B\cup D\subseteq A'\cup B'$ by \ref{clm:c5-c5-dd'}. First we suppose that $D\cap B_i'=\emptyset$ for some $i\in \langle 5 \rangle $. Since $B_i'$ is a stable set, by \ref{clm:c5-bstable} and \ref{clm:c5-bb+1}, $B_i'\subseteq B_j\cup B_{j+2}$ for some $j\in \langle 5 \rangle $. Then $(B_{j-2}\cup B_{j-1})\cap B_i'=\emptyset$, and then $(D\cup B_{j-2}\cup B_{j-1})\cap B_i'=\emptyset$. By \ref{clm:c5-3or4colorable}, $G[D\cup B_{j-2}\cup B_{j-1}]$ is $4$-colorable, and $G-(D\cup B_{j-2}\cup B_{j-1})=G[A\cup B_j\cup B_{j+1}\cup B_{j+2}]$ is $3$-colorable, $G$ is $7$-colorable.

So we may assume that $D\cap B_i'\neq\emptyset$, for each $i\in \langle 5 \rangle $. Let:
\begin{eqnarray*}
{\mathbb F}_1&:=&\{X\in\{B_1,\ldots,B_5\}\mid X\subseteq B_i', \mbox{ for some } i\in \langle 5 \rangle \}, \mbox{ and}\\
{\mathbb F}_2&:=&\{X\in\{B_1,\ldots,B_5\}\mid X\nsubseteq B_i',  \mbox{ for any } i\in \langle 5 \rangle \}.
\end{eqnarray*}
Then there exists some $i\in \langle 5 \rangle $ such that $B_i,B_{i+1}\in {\mathbb F}_1$ or $B_i,B_{i+1}\in {\mathbb F}_2$, say $i=1$. Let $X:=D\cup B_1\cup B_2$. First suppose that $B_1,B_2\in {\mathbb F}_2$. Then, up to symmetry, let $B_1\subseteq B_1'\cup B_3'$ and $B_2\subseteq B_2'\cup B_4'$. Since $B\cap B_2'\neq\emptyset$, by \ref{clm:c5-c5-bbcompdb}, $D\cap B_1'$ and $D\cap B_3'$ are anti-complete. There exist vertices, say $b_1\in B_1\cap B_1'$, $b_2\in B_1\cap B_3'$, $d_1\in D\cap B_1'$ and $d_2\in D\cap B_3'$ such that $d_1d_2\notin E$. By \ref{clm:c5-bstable}, $b_1b_2,b_1d_1,b_2d_2\notin E$. By \ref{clm:c5-c5-bbcompdb}, $b_1d_2,b_2d_1\in E$. Then $\{b_1,b_2,d_1,d_2\}$ induces a $2K_2$.
So, we may assume that $B_1,B_2\in {\mathbb F}_1$. By symmetry let $B_1\subseteq B_1'$, and $B_2\subseteq B_2'$ or $B_3'$. Since each $D\cap B_i'$ is nonempty, by \ref{clm:c5-c5-bbcompdb}, $D\cap B_2'$ and $D\cap B_4'$, $D\cap B_3'$ and $D\cap B_5'$ are anti-complete. If $B_2\subseteq B_2'$, then $X\cap B_3'$ and $X\cap B_5'$ are anti-complete. Likewise, if $B_2\subseteq B_3'$, then $X\cap B_2'$ and $X\cap B_4'$ are anti-complete.
By \ref{clm:c5-3or4colorable}, $G[X]$ is 4-colorable , while $G-X=G[A\cup B_3\cup B_4\cup B_5]$ is 3-colorable. So $G$ is 7-colorable.

This completes the proof of \cref{thm:ghasc5} in \cref{C5-case1}. $\Diamond$

\begin{case}\label{C5-case2}
$G[D]$ is $C_5$-free and contains a $\overline{C_6}$.\end{case}
Suppose that $G[D]$ contains a $\overline{C_6}$ with vertex-set $R:=\{r_1,r_2, \ldots,r_6\}$ such that $r_jr_{j+1}\notin E$, $j\in \langle 6 \rangle $. For $j\in \langle 6 \rangle $, let:
\begin{eqnarray*}
S_j&:=&\{u\in D\setminus R\mid N(u)\cap R=R\sm \{r_{j-1},r_{j-2}\}\}, \mbox{ and }\\
T_j&:=&\{u\in D\setminus R\mid N(u)\cap R=\{r_{j-1},r_j,r_{j+1}\}\}.
\end{eqnarray*}
Let $S:=\cup_{j=1}^6S_j$  and $T:=\cup_{j=1}^6T_j$. Note that, by \ref{clm:c5-bnd} and \ref{clm:c5-bdk3}, every vertex in $B$ belongs to some $W_j$, where $j\in \langle 6 \rangle $ as defined below:
$$W_j:=\{b\in B\mid N(u)\cap R=R\sm \{r_{j-1},r_{j-2}\}\}.$$
So $B=\cup_{j=1}^6W_j$. Recall that, by \ref{clm:c5-dk4}, $G[D]$ is ($K_4,C_5$)-free. Moreover, we see that the following properties hold, where $j\in \langle 6 \rangle $.

\begin{enumerate} [label=(\roman*)]
\item\label{c5frc6bar-Dpart} $D=R\cup S\cup T$.

\begin{proof2}
 Suppose that there is a vertex $u\in D\sm (R\cup S\cup T)$.  First assume that there is an index $k\in \langle 6 \rangle $ such that $\{r_k,r_{k+1}, r_{k+2}\}\subseteq N(u)\cap R$. Then since $\{u, r_k,r_{k+2}, r_{k-2}\}$ does not induce a $K_4$ in $G[D]$, we have $ur_{k-2}\notin E$, and then since $\{r_{k-1},r_{k-2},r_{k+3},u\}$ does not induce a $K_3+K_1$, one of $ur_{k-1}\notin E$ and $ur_{k+2}\notin E$. Thus, we observe that $u\in S\cup T$, a contradiction. So we may assume that $u$ is not complete to any three consecutive vertices of $R$.
By \ref{dom-k3}, $u$ has a neighbor in $R$, say $r_j$. Then since $\{u,r_j, r_{j-1}, r_{j+1}\}$ does not induce a $2K_2$, we may assume that $ur_{j+1}\in E$. By our assumption, $ur_{j-1}, ur_{j+2}\notin E$. Now $\{u,r_j,r_{j+2},r_{j-1},r_{j+1}\}$ induces a $C_5$ in $G[D]$, a contradiction. This proves \ref{c5frc6bar-Dpart}.
\end{proof2}

\item\label{c5-c6bar-wstable} Each of $S_j$, $T_j$ and $W_j$ is a stable set.

\begin{proof2}
Suppose that there are adjacent vertices, say $x$ and $y$ in one of the stated lists. Now, if $x,y\in S_j$ are adjacent, then $\{x,y,r_j,r_{j+2}\}$ induces a $K_4$ in $G[D]$. If $x,y\in T_j$, then $\{x,y,r_{j-1},r_{j+1}\}$ induces a $K_4$ in $G[D]$. Finally, if $x, y\in W_j$, then \{$x,y,r_j,r_{j-1}$\} induces a $K_3+K_1$.
\end{proof2}

\item\label{c5-c6bar-ww+2} $W_j$ and $W_{j+2}\cup W_{j-2}$ are complete. So, for each $i\in \langle 5 \rangle $, $B_i$ is a subset of $W_j\cup W_{j+1}$ or $W_j\cup W_{j+3}$ for some $j\in \langle 6 \rangle $.

\begin{proof2}
If there are nonadjacent vertices, say, $b\in W_j$ and $b_2\in W_{j+2}$, then \{$b,r_j,b',r_{j-1}$\} induces a $2K_2$.
This proves the first assertion. Since $B_i$ is a stable set (by \ref{clm:c5-bstable}), $B_i$ has vertices in at most one of $W_1,W_3,W_5$ and in at most one of $W_2,W_4,W_6$.
\end{proof2}

\item\label{c5-c6bar-bb+2} Suppose there are adjacent vertices, say $b\in B_i\cap W_j$ and $b'\in B_{i+2}$. Then $b'\in W_{j+2}\cup W_{j+3}\cup W_{j+4}$.

\begin{proof2}
By \ref{clm:c5-bb+2domi}, \{$b,b'$\} dominates $R$. So $R\setminus N(b)=\{r_{i-2},r_{i-1}\}\subseteq N(b')$.
\end{proof2}

\item\label{c5-c6bar-sw} Suppose that  $S_j\neq \es$. Then $B= W_j\cup W_{j+2}\cup W_{j+4}$. Moreover, $S_j$ is complete to $W_{j+2}\cup W_{j+4}$, and anti-complete to $W_j$.

\begin{proof2}
Let $s\in S_j$. Then note that each of \{$s,r_j,r_{j+2}$\}, \{$s,r_{j+1},r_{j+3}$\} and \{$s,r_j,r_{j+3}$\} induces a $K_3$ in $D$.
 If there is a vertex $b\in W_{j+1}$, then, by \ref{clm:c5-bdk3}, $bs\in E$, and so $b$ is complete to $\{ s,r_{j+1},r_{j+3}\}$, a contradiction to \ref{c5:dbk3-free}; so $W_{j+1}=\es$. Likewise, $W_{j-1}=\es$. Finally, if
there is a vertex $b\in W_{j+3}$, then, again, by \ref{clm:c5-bdk3},  $b$ is complete to $\{ s,r_{j},r_{j+3}\}$, a contradiction to \ref{c5:dbk3-free}; so $W_{j+3}=\es$. This proves $B= W_j\cup W_{j+2}\cup W_{j+4}$.
The second assertion follows from   \ref{clm:c5-bdk3} and \ref{c5:dbk3-free}.
\end{proof2}

\item\label{c5-c6bar-tw} $T_j$ is complete to $W_j\cup W_{j+1}\cup W_{j+2}\cup W_{j+3}$ and anti-complete to $W_{j-2}\cup W_{j-1}$.

\begin{proof2}
Since for any $t\in T_j$, $\{t,r_{j-1},r_{j+1}\}$ induces a $K_3$ in $G[D]$, by \ref{clm:c5-bdk3}, for any $b\in B$, $|N(b)\cap \{t,r_{j-1},r_{j+1}\}|=2$. Now the proof follows from the definition of $W_j$'s.
\end{proof2}

\item\label{c5-c6bar-tt+1}If there are adjacent vertices, say, $t\in T_j$ and $t'\in T_{j+1}$, then $B=W_j\cup W_{j+2}\cup W_{j+4}$.

\begin{proof2} If $W_{j-1}\neq \es$, then since $\{t, r_{j-1}, r_{j+1}\}$ and $\{t',r_{j}, r_{j+2}\}$ induce triangles in $G[D]$, by \ref{clm:c5-bdk3} and by the definition of $W_{j-1}$, \{$t,t'$\} is anti-complete to $W_{j-1}$, and then, for any $b\in W_{j-1}$, $\{t,t'\}\subseteq D\sm N(b)$, a contradiction to \ref{clm:c5-bnd}; so   $W_{j-1}=\emptyset$. Also, since \{$t_1,t_2,r_{j+1}$\} induces a $K_3$ in $D$ and is complete to $W_{j+3}$ (by \ref{c5-c6bar-tw}), we have $W_{j+3}=\emptyset$ (by \ref{clm:c5-bdk3}). Likewise, $W_{j+1}=\emptyset$.
\end{proof2}

\item\label{c5-c6bar-wd3c} $G[W_j\cup D]$ is 3-colorable.

\begin{proof2}
Suppose that $G[W_j\cup D]$ is not 3-colorable. Since $G[W_j\cup D]$ is $K_4$-free (by \ref{c5:dbk3-free}), $G[W_j\cup D]$ is imperfect, and hence by \cref{thm:SPGT}, $G[W_j\cup D]$ contains an odd hole or an odd anti-hole, say $C$. Since $G[D]$ is (odd hole, odd anti-hole)-free, $V(C)\cap W_j\neq \es$. By \ref{c5-c6bar-sw} and \ref{c5-c6bar-tw}, if $b$ and $b'$ are any two vertices in $W_j$, then $N(b)\cap D = N(b')\cap D$. So $C$ has exactly one vertex in $W_j$, which contradicts \ref{clm:c5-bnd}.
\end{proof2}
\end{enumerate}
Now we proceed to prove the lemma in this case. If $B_i\subseteq W_j$ for some $i\in \langle 5 \rangle $ and $j\in \langle 6 \rangle $, then $G[B_i\cup D]$ is 3-colorable (by \ref{c5-c6bar-wd3c}),  $G[A\cup(B\setminus B_i)]$ is 4-colorable (by using \ref{c5-ab-stable}), and hence $G$ is 7-colorable. So we may assume that $B_i\nsubseteq W_j$ for any $i\in \langle 5 \rangle $ and $j\in \langle 6 \rangle $. Then by \ref{c5-c6bar-ww+2} and \ref{c5-c6bar-sw}, $S=\emptyset$. By \ref{c5-c6bar-tt+1}, $T_j$ and $T_{j+1}$ are anti-complete for $j\in \langle 6 \rangle $. Clearly,  by \ref{c5-c6bar-wstable} and \ref{c5-c6bar-tw} and by the definitions of $T$ and $W_j$'s, $\{r_1,r_2\}\cup T_4\cup T_5$, $\{r_4,r_5\}\cup T_1\cup T_2\cup W_6$, $\{r_3\}\cup T_6\cup W_4\cup W_5$, and $\{r_6\}\cup T_3\cup W_1$  are four stable sets, and hence $H_1:=G[W_1\cup W_4\cup W_5\cup W_6\cup D]$ is $4$-colorable. Let $H_2:=G-H_1=G[A\cup W_2\cup W_3]$. Then $V(H_2)\cap B_i=\emptyset$ for some $i\in \langle 5 \rangle $. By \ref{c5-c6bar-bb+2}, $V(H_2)\cap B_{i-1}$ and $V(H_2)\cap B_{i+1}$ are anti-complete. By \ref{clm:c5-3or4colorable}, $H_2$ is 3-colorable. Then since $\chi(G)\leq \chi(H_1)+\chi(H_2) \leq 7$, $G$ is $7$-colorable.

 This completes the proof of \cref{thm:ghasc5} in \cref{C5-case2}. $\Diamond$

\begin{case}\label{c5-case3}
$G[D]$ is ($C_5$, $\overline{C_6}$)-free.\end{case}
First suppose that $G[D]$ is $K_3$-free. Then $G[D]$ is a bipartite graph therefore is 2-colorable.
 Also, by \ref{c5-ab-stable},  $A_{i+1}\cup B_i$ is a stable set, for each $i\in \langle 5 \rangle $, and hence $G[A\cup B]$ is 5-colorable. So  we conclude that $G$ is 7-colorable.  So suppose that $G[D]$ contains a $K_3$ with vertex-set, say $R:=\{r_1,r_2,r_3\}$. For $j\in \langle 3 \rangle $, let:
\begin{eqnarray*}
 S_j &:=& \{u\in D\setminus R \mid N(u)\cap R =\{r_j\}\}, \mbox{ and} \\
T_j &:=&\{u\in D\setminus R \mid N(u)\cap R =\{r_{j+1},r_{j-1}\}\}.
\end{eqnarray*}
Let $S:=S_1\cup S_2\cup S_3$ and $T:=T_1\cup T_2\cup T_3$. Then, since $G$ is $K_3+K_1$-free, every triangle in $G[D]$ is dominating; so  each $u\in D\setminus R$ has a neighbor in $R$, and since $G[D]$ is $K_4$-free (by \ref{clm:c5-dk4}), by \ref{nei-Kfree}, we conclude that $|N(u)\cap R|$ is either $1$ or $2$, and hence $D=R\cup S\cup T$. Also, by \ref{clm:c5-bdk3}, every vertex in $B$ belongs to some $W_j$ as defined below, where $j\in \langle 3 \rangle $:
$$W_j:=\{b\in B\mid N(u)\cap R=\{r_{j+1},r_{j-1}\}\}.$$
More precisely, $B=W_1\cup W_2\cup W_3$. Now we see that the following properties hold, where $j\in \langle 3 \rangle $.

\begin{enumerate}[label=(\roman*),series=edu]
 \item\label{c5-free-stable} Each of $S_j$ and $T_j$ is a stable set.

\begin{proof2}
Suppose that  there are adjacent vertices, say $x$ and $y$ in $G$. If $x,y\in S_i$ are adjacent, then \{$x,y,r_{j+1},r_{j-1}$\} induces a $2K_2$. If $x,y\in T_j$, then \{$x,y,r_{j+1},r_{j-1}$\} induces a $K_4$.
\end{proof2}

\item\label{c5-free-wbipartite} Every $W_j$ induces a bipartite graph. This implies that $W_j\cap B_i=\emptyset$ for some $i\in \langle 5 \rangle $.

\begin{proof2}
Clearly $W_j\subseteq\overline N(r_j)$. So,  by \ref{nonnei-bip}, $W_j$ induces a bipartite graph. This proves the first assertion. Suppose that there are vertices, say, $b_1,b_2,b_3,b_4$ and $b_5$ such that $b_i\in W_j\cap B_i$  for each $i\in \langle 5 \rangle$. Then, by \ref{clm:c5-bb+1}, $b_ib_{i+1}\in E$ for each $i\in \langle 5 \rangle$. If $b_ib_{i+2}\in E$ for some $i\in \langle 5 \rangle$, then $\{b_i,b_{i+1},b_{i+2}\}$ induces a $K_3$ in $G[W_j]$. If $b_ib_{i+2}\notin E$ for each $i\in \langle 5 \rangle$, then $\{b_1,b_2,b_3,b_4,b_5\}$ induces a $C_5$ in $G[W_j]$. Both cases contradicts the first assertion.
%
\end{proof2}

\item\label{c5-free-wbwb+2} For $i\in \langle 5 \rangle $, $W_j\cap B_i$ and $W_j\cap B_{i+2}$ are anti-complete. Moreover, $W_j\cap (B_i\cup B_{i+2})$ is a stable set.

\begin{proof2}
If there are adjacent vertices, say $b\in W_j\cap B_i$ and $b'\in W_j\cap B_{i+2}$, then $br_j,b'r_j\notin E$, which contradicts \ref{clm:c5-bb+2domi}. This proves the first assertion. The second assertion follows from the first assertion, and from \ref{clm:c5-bstable}.
\end{proof2}

\item\label{c5-free-stsw} $S_j$ and $T_j\cup W_j$ are complete.

\begin{proof2}
If there are nonadjacent vertices, say $s\in S_j$ and $t\in T_j\cup W_j$, then \{$t,r_{j+1},r_{j-1},s$\} induces a $K_3+K_1$.
\end{proof2}

\item\label{c5-free-tw} $T_j$ and $W_j$ are anti-complete.

\begin{proof2}
If there are adjacent vertices, say $t\in T_j$ and $b\in W_j$, then \{$t,r_{j+1},r_{j-1}$\} induces a $K_3$ in $N(b)\cap D$, which contradicts \ref{clm:c5-bdk3}.
\end{proof2}

\item\label{c5-free-semptystable}  Either $S_j$ is empty for some $j\in \langle 3 \rangle $, or $S$ is a stable set.

\begin{proof2}
Suppose that there are vertices, say $s_1\in S_1$, $s_2\in S_2$ and $s_3\in S_3$, and that $s_1s_2\in E$. If $s_1s_3,s_2s_3\notin E$, then \{$s_1,s_2,r_3,s_3$\} induces a $2K_2$. If $s_1s_3,s_2s_3\in E$, then \{$s_1,s_2,s_3,r_1,r_2,r_3$\} induces a $\overline{C_6}$. If $s_3$ is adjacent to exactly one of $s_1$ and $s_2$, say $s_1$, then \{$s_1,s_2,r_2,r_3,s_3$\} induces a $C_5$.
\end{proof2}

\item\label{c5-free-swcomplete} Suppose that $S_j\neq\emptyset$, for each $j\in \langle 3 \rangle $. If there are nonadjacent vertices, say $b\in W_{j+1}$ and $b'\in W_{j-1}$, then \{$b,b'$\} is complete to $S$.

\begin{proof2}
Since  $S_j\neq\emptyset$, for each $j\in \langle 3 \rangle $, there are vertices, say, $s_1\in S_{j+1}$, $s_2\in S_{j-1}$, and $s_3\in S_j$. By \ref{c5-free-stsw}, $s_1b,s_2b'\in E$. Then since \{$s_1,b,s_2,b'$\} does not induce a $2K_2$, we have either $s_1b'\in E$ or $s_2b\in E$.  By symmetry, assume that $s_1b'\in E$.  Then since \{$b',s_1,r_{j+1},s_3$\} does not induce a $K_3+K_1$, $s_3b'\in E$, and then since \{$s_3,b',b,r_{j-1}$\} does not induce a $2K_2$,  we have $s_3b\in E$. Finally, since \{$s_2,b,s_3,r_j$\} does not induce a $K_3+K_1$, we have $s_2b\in E$. This proves \ref{c5-free-swcomplete}.
\end{proof2}

\end{enumerate}
To prove the lemma in this case, we consider two cases:

\medskip
\no{\bf Case 3.1}~ {\em $S_j=\emptyset$, for some $j\in \langle 3 \rangle $.}

\smallskip
\no We may assume that $S_3=\emptyset$. By \ref{c5-free-wbipartite}, we may also assume that $W_1\cap B_1=\emptyset$. Let $H_1$ denote the induced subgraph $G[D\cup(W_1\cap(B_2\cup B_5))\cup(W_2\cap B_1)\cup W_3]$ and let $H_2$ denote the induced subgraph $G-V(H_1)= G[A\cup(W_1\cap(B_3\cup B_4))\cup(W_2\cap(B\setminus B_1))]$.

 Now we show that $\chi(H_1)\leq 4$. First, since $X_1:=T_3\cup W_3\cup S_1\cup S_2\subseteq \overline N(r_3)$, by \ref{nonnei-bip}, $X_1$ induces a bipartite graph. Next,   by \ref{c5-free-wbwb+2}, $W_1\cap(B_2\cup B_5)$ is a stable set, and hence, by \ref{c5-free-tw} and by the definitions of $T$ and $W,$  $X_2:= \{r_1\}\cup T_1\cup (W_1\cap(B_2\cup B_5))$ is a stable set. Likewise, $X_3:=\{r_2\}\cup T_2\cup (W_2\cap B_1)$ is a stable set. Then clearly, $V(H_1)=X_1\cup X_2\cup X_3$, and hence $\chi(H_1)\leq 4$.

 Next, we show that $\chi(H_2)\leq 3$. Since $V(H_2)\cap (B_2\cup B_5)\subseteq W_2$, by \ref{c5-free-wbwb+2}, $V(H_2)\cap B_2$ and $V(H_2)\cap B_5$ are anti-complete. Then since $V(H_2)\cap B_1=\emptyset$, by \ref{clm:c5-3or4colorable}, $\chi(H_2)\leq 3$.

 Thus, $\chi(G)\leq \chi(H_1)+\chi(H_2)\leq 7$, and hence $G$ is $7$-colorable.

\medskip
\no{\bf Case 3.2}~ {\em $S_j\neq \emptyset$, for each $j\in \langle 3 \rangle $.}

\smallskip

\no Recall that, by \ref{c5-free-semptystable},   $S$ is a stable set.

First suppose that $B_i\subseteq W_j$ for some $i\in \langle 5 \rangle $ and $j\in \langle 3 \rangle $. Let $B_1\subseteq W_3$. Let $H_1$ denote the induced subgraph $G[D\cup((W_1\cup W_2)\cap(B_2\cup B_5))\cup B_1]$ and let $H_2$ denote the induced subgraph  $G-H_1=G[A\cup(W_3\cap(B_2\cup B_5))\cup((W_1\cup W_2\cup W_3)\cap(B_3\cup B_4))]$.  Then, as in Case~3.1, since $B_1\subseteq W_3$,  we see that  $\{r_1\}\cup T_1\cup(W_1\cap(B_2\cup B_5))$, $\{r_2\}\cup T_2\cup(W_2\cap(B_2\cup B_5))$ and $\{r_3\}\cup T_3\cup B_1$ are three stable sets. Also, since $S$ is a stable set, we have $\chi(H_1)\leq 4$. Next, since $V(H_2)\cap B_2$ and $V(H_2)\cap B_5$ are anti-complete by \ref{c5-free-wbwb+2}, and since $V(H_2)\cap B_1=\emptyset$, by \ref{clm:c5-3or4colorable}, $\chi(H_2)\leq 3$. Then $\chi(G)\leq \chi(H_1)+\chi(H_2) \leq 7$, and hence $G$ is $7$-colorable.

So we may assume that $B_i\nsubseteq W_j$ for any $i\in \langle 5 \rangle $ and $j\in \langle 3 \rangle $. By \ref{c5-free-swcomplete} and \ref{clm:c5-bstable}, $B$ and $S$ are complete. Now if there are adjacent vertices, say, $b,b'\in W_j$, then for any $s\in S_{j+1}$, \{$b,b',s,r_j$\} induces a $K_3+K_1$; so $W_j$ is a stable set, for $j\in \langle 3 \rangle $. Hence for $j\in \langle 3 \rangle $, by \ref{c5-free-tw}, $\{r_j\}\cup T_j\cup W_j$ is a stable set.  Since $S$ is a stable set,  $G[D\cup B] = G[R\cup S\cup T\cup W_1\cup W_2\cup W_3]$ is 4-colorable. Then since $G[A]$ is 3-colorable, we conclude that $G$ is 7-colorable.

This completes the proof of \cref{thm:ghasc5} in \cref{c5-case3}, and hence the proof of  \cref{thm:ghasc5}. \hfill{$\Box$}

\bigskip
\no{\bf Acknowledgement}.  The third author would like to thank Christoph Brause for personally communicating his work \cite{BG21}.

\end{document}